\documentclass[final,1p,times]{elsarticle}
\usepackage{amsthm}
\usepackage{amscd}
\usepackage{amsmath}
\usepackage{amsfonts}
\usepackage{amssymb}
\usepackage{graphicx}
\newtheorem{theorem}{Theorem}[section]

\newtheorem{lemma}[theorem]{Lemma}
\newtheorem{remark}{Remark}

\usepackage{mathrsfs}
\usepackage{titletoc}
\numberwithin{equation}{section}
\newcommand{\rw}{\rightarrow}

\renewcommand{\d}{\delta}

\newcommand{\D}{\Delta}
\newcommand{\ra}{\rightarrow}

\newcommand{\f}{\frac}
\newcommand{\al}{\alpha}

\renewcommand{\l}{\lambda}
\renewcommand{\L}{\Lambda}
\newcommand{\be}{\begin{equation}}
\renewcommand{\ra}{\rightarrow}
\newcommand{\ee}{\end{equation}}
\newcommand{\bea}{\begin{eqnarray}}
\newcommand{\eea}{\end{eqnarray}}
\newcommand{\bna}{\begin{eqnarray*}}
\newcommand{\ena}{\end{eqnarray*}}
\renewcommand{\o}{\omega}

\newcommand{\iv}{\int_{V}}

\renewcommand{\le}{\left}
\newcommand{\ri}{\right}

\newcommand{\ve}{\vert}
\newcommand{\V}{\Vert}

\newcommand{\na}{\nabla}
\newcommand{\upl}{\upsilon'_{\lambda}}
\newcommand{\up}{\upsilon}

\journal{***}
\begin{document}

\begin{frontmatter}

\title{ Existence of solutions to Chern-Simons-Higgs equations on graphs}

\author{Songbo Hou \corref{cor1}}
\ead{housb@cau.edu.cn}
\address{Department of Applied Mathematics, College of Science, China Agricultural University,  Beijing, 100083, P.R. China}
\author{Jiamin Sun}
\ead{1416525364@qq.com}
\address{Department of Applied Mathematics, College of Science, China Agricultural University,  Beijing, 100083, P.R. China}

\cortext[cor1]{Corresponding author: Songbo Hou}

\begin{abstract}
Let $G=(V,E)$ be a finite graph.  We  consider the existence of solutions to a generalized Chern-Simons-Higgs equation
$$
\Delta u=-\lambda e^{g(u)}\left( e^{g(u)}-1\right)^2+4\pi\sum\limits_{j=1}^{N}\delta_{p_j}
$$
on $G$, where $\lambda$ is a positive constant; $g(u)$ is the inverse function of $u=f(\upsilon)=1+\upsilon-e^{\upsilon}$ on $(-\infty, 0]$; $N$ is a positive integer;  $p_1, p_2, \cdot\cdot\cdot, p_N$ are distinct vertices of $V$
and $\delta_{p_j}$ is  the Dirac delta mass at $p_j$.
We prove that there is critical value  $\lambda_c$ such that the generalized Chern-Simons-Higgs equation has a solution if and only if $\l\geq \l_c$ . We also prove the existence of solutions to the Chern-Simons-Higgs equation
 $$
\Delta u=\lambda e^{u}(e^{u}-1)+4\pi\sum\limits_{j=1}^{N}\delta_{p_j}
$$
on $G$ when $\lambda$ takes the critical value $\lambda_c$ and this completes the results of An Huang, Yong Lin and Shing-Tung Yau (Commun. Math. Phys. 377, 613-621 (2020)).

\end{abstract}


 \end{frontmatter}

\section{Introduction}

The PDE  is an important tool in various applications on graphs, such as filtering, denoising, segmentation
and clustering, et al. \cite{AE12, LJ10}. The study of PDEs on graphs has become an interesting subject.

We recall the definition of a graph.  Let $V$ be the vertex set and $E$ be the edge set. We use  $G=(V,E)$ to denote a finite graph. We assume that for each edge $xy\in E$, its weight $\o_{xy}$ satisfies $\o_{xy}>0$ and $\o_{xy}=\o_{yx}$.
For any $x,y\in E$, we assume that they can be connected via finite edges and then $G$ is called a connected graph. Let $\mu :V\ra \mathbb{R}^{+}$ be a finite measure. For any function $u:V\ra \mathbb{R}$, the
$\mu$-Laplace is defined by
$$\D u(x)=\f{1}{\mu(x)}\sum\limits_{y\sim x}\omega_{xy}(u(y)-u(x)),$$ where $y\sim x$ means $xy\in E$.  For a pair of functions $u$ and $\up$, the gradient form is defined by
$$\Gamma(u,\up)=\f{1}{2\mu(x)}\sum\limits_{y\sim x}\omega_{xy}(u(y)-u(x))(\up(y)-\up(x)).$$
When $u=\up$, we write $\ve \na u\ve^2=\Gamma(u,u)$. For a function $u:V\ra \mathbb{R}$, we denote the integral over $V$ by
$$\int_V ud\mu=\sum\limits_{x\in V}\mu(x)u(x).$$   In order to study the PDEs on graphs, we also need to define the Sobolev space and the norm.  Let $W^{1,2}(V)$ be the space of functions $u:V\rw \mathbb{R}$ satisfying

$$\iv\le(\ve \na u\ve^2 +u^2\ri)d\mu<+\infty.$$
The norm of $u$ in $W^{1,2}(V)$ is defined as
$$\V u\V_{W^{1,2}(V)}=\le(\iv\le(\ve \na u\ve^2 +u^2\ri)d\mu\ri)^{1/2}.$$

Recently, the methods to deal with the PDEs in the Euclidean settings were performed on graphs successfully.  Grigor'yan, Lin and  Yang \cite{GLY16, GLY17} used the variation method and  the mountain-pass theorem to obtain the  existence results for some nonlinear equations on gaphs. The similar method was used by the first author to prove the existence of multiple solutions to a nonlinear biharmonic equation on graphs \cite{Hou21}.
Using the the Nehari method,  Zhang and Zhao \cite{ZZ18} studied the convergence of ground state solutions for nonlinear
Schr\"{o}dinger equations on graphs. The Nehari method was also used by Han, Shao and Zhao \cite{HSZ}  to obtain  the existence and convergence of solutions for nonlinear biharmonic equations on graphs. For the Kazdan-Warner equation, using the calculus of variations and a method of upper and lower solutions, Grigor'yan, Lin and  Yang \cite{GLY} obtained the existence of solutions under various conditions. Keller and Schwarz \cite{KS18} studied the Kazdan-Warner equation on canonically compactifiable graphs and proved the existence results under various conditions.  Liu and Yang \cite{LY20} studied the multiple solutions of Kazdan-Warner equation  in the negative case.  The existence of the solutions to the  $p$-th Kazdan-Warner equation was studied in \cite{Ge20, ZC18}. Related results also include \cite{GJ18, GHJ, LW18, LW17, ZL19} and so on.

 There are a lot of study about vortices in (2 + 1)-dimensional Chern-Simons gauge theory in recent years. Burzlaff, Chakrabarti and Tchrakian \cite{BCT92}  proposed certain generalisations  of  the self-dual Chern-Simons-Higgs model \cite{HKP90, JW90}, whereas the existence of double periodic vortices for the models
had been an open problem.
  Han  \cite{Han14} reduced the generalized
self-dual Chern-Simons-Higgs equation to  a semilinear elliptic equation. Then  he solved the existence problem by an upper-lower solution method.

In this paper, we first consider the  generalized Chern-Simons-Higgs equation derived by Han \cite{Han14}, i.e.,
\be\label{cs}
\D u=-\l e^{g(u)}\le( e^{g(u)}-1\ri)^2+4\pi\sum\limits_{j=1}^{N}\delta_{p_j}
\ee
on a graph $G=(V,E)$ where $\l$ is a positive constant; $g(u)$ is the inverse function of $u=f(\up)=1+\up-e^{\up}$ on $(-\infty, 0]$; $N$ is a positive integer;  $p_1, p_2, \cdot\cdot\cdot, p_N$ are distinct vertices of $V$
and $\d_{p_j}$ is  the Dirac delta mass at $p_j$.

We prove the following existence result for Eq.(\ref{cs} ) on $G$.
\begin{theorem}
There exists a critical value $\l_c$ depending on the graph satisfying
$$\l_c\geq \f{27\pi N}{\ve V\ve},$$
such that if $\l\geq \l_c$, Eq.(\ref{cs}) has a negative solution , and if $\l<\l_c$,  Eq.(\ref{cs}) has no solution.
\end{theorem}

Secondly, we consider the Chern-Simons-Higgs equation
\be\label{CSH} \D u=\l e^u(e^u-1)+4\pi\sum\limits_{j=1}^{N}\delta_{p_j},\ee
 which was proposed in \cite{CY95}. Huang, Lin and Yau \cite{HLY} prove that on a finite graph $G$, there is a critical $\l_c$ satisfying
 $$\l_c\geq \f{16\pi N}{\ve V\ve}, $$ such that  Eq.(\ref{CSH}) has a solution if $\l>\l_c$, and  Eq.(\ref{CSH}) has no solution if $\l<\l_c$.

 We study  Eq.(\ref{CSH}) when $\l=\l_c$ and get the following result.
 \begin{theorem}
  Eq.(\ref{CSH}) has a solution when $\l=\l_c$ on a finite graph $G$.
 \end{theorem}
 In the proof of the main theorems, we take the similar treatments as in the continuous case \cite{CY95, Han14, GT96}. The rest of the paper is arranged as follows. In Section Two, we give some preliminaries which will be used later. In Section Three, we prove Theorem 1.1. In Section Four, we prove Theorem 1.2.
\begin{remark}
L\"{u} and Zhong \cite{LZ} obtained  the existence of the single solution to Eq.(\ref{cs}) by a different method. They also studied the existence of multiple
solutions.
\end{remark}

\section {Preliminaries }
In this section, we introduce three lemmas which will be used in Section Three and Section Four.
\begin{lemma} (Huang-Lin-Yau \cite{HLY})

Let $G=(V,E)$ be a finite graph.
If for some constant $K>0$, $\D  u(x) -Ku(x) \geq 0$ for all $x\in V$, then $u\leq 0$ on $V$ .
\end{lemma}

\begin{lemma}(Grigor'yan-Lin-Yang \cite{GLY})

Let $G=(V,E)$ be a finite graph. Then there exists some constant $C$ depending only on $G$ such that
$$\iv u^2d\mu\leq C\iv \ve \na u\ve^2d\mu,$$ for all $u\in V^{\mathbb{R}}$ with $\iv ud\mu=0$, where $V^{\mathbb{R}}=\{u\,|\,u \,\text{is a real function}: V\ra\mathbb{R}\}$.
 \end{lemma}
\begin{lemma}\label{lm3}  (Grigor'yan-Lin-Yang \cite{GLY})

Let $G=(V,E)$ be a finite graph. For any $\al>0$ and all functions $u\in V^{\mathbb{R}}$ with $\iv \ve \na u\ve^2 d\mu\leq 1$ and $\iv ud\mu=0$, there exists some constant $C>0$ depending only on $\al$
and $G$ such that
$$\iv e^{\al u^2}d\mu\leq C.$$
\end{lemma}

\section{Proof of  Theorem 1.1}
Noting
$$\iv \le(-\f{4\pi N}{\ve V\ve}+4\pi \sum\limits_{j=1}^{N}\delta_{p_j}\ri)d\mu=0,$$ we assert that there exists a solution $\upsilon_0$ of the equation
\be \label{up1}\D \up_0= -\f{4\pi N}{\ve V\ve}+4\pi \sum\limits_{j=1}^{N}\delta_{p_j}. \ee

Letting $u=\up_0+\up$, then we rewrite Eq.(\ref{cs}) as
\be\label{cs2}
\D \up =-\l e^{g(\up_0+\up)}\le(e^{g(\up_0+\up)}-1\ri)^2+\f{4\pi N}{\ve V\ve}.
\ee
We will use an upper-lower solution method to derive the existence results of Eq.(\ref{cs2}). A function $\up_-$ is called a lower solution of (\ref{cs2}) if  it satisfies
$$\D \up_- \geq -\l e^{g(\up_0+\up_-)}\le(e^{g(\up_0+\up_-)}-1\ri)^2+\f{4\pi N}{\ve V\ve},$$
for all $x\in V$.
Similarly, a function $\up_+$ is called an upper solution of (\ref{cs2}) if  it satisfies
$$\D \up_+ \leq -\l e^{g(\up_0+\up_+)}\le(e^{g(\up_0+\up_+)}-1\ri)^2+\f{4\pi N}{\ve V\ve},$$
for all $x\in V$.
 Let $\psi_0=-\up_0$. Obviously, $\psi_0$
is an upper solution of (\ref{cs2}). We introduce the following iterative sequence $\{\psi_n\}$,

\be\label{cs3}\le(\D-K\ri)\psi_n=-\l e^{g(\up_0+\psi_{n-1})}\le(e^{g(\up_0+\psi_{n-1})}-1\ri)^2-K\psi_{n-1}+\f{4\pi N}{\ve V\ve},\ee
where $K$ is a positive constant and $n=1,2,\cdot\cdot\cdot.$
\begin{lemma}
Suppose $K>2\l$. Let $\{\psi_n\}$ be the sequence defined by (\ref{cs3}). For any lower solution $\up_-$ of (\ref{cs2}), there holds
$$\psi_0> \psi_1> \psi_2>\cdot\cdot\cdot> \psi_n>\cdot\cdot\cdot
>\up_{-}.$$ Therefore,  if (\ref{cs2}) has a lower solution, then $\{\psi_n\}$ converges to a maximal solution of (\ref{cs2}).
 \end{lemma}
 \begin{proof}
 We proceed by induction. For $n=1$, we have
 $$\le(\D-K\ri)\psi_1=-K\psi_0+\f{4\pi N}{\ve V\ve}.$$  It follows that
 \be\label{ps1}\le(\D-K\ri)(\psi_1-\psi_0)=4\pi \sum\limits_{j=1}^{N}\delta_{p_j}\geq 0.\ee Then Lemma 2.1 implies that
 $\psi_1\leq \psi_0$ on $V$. We claim that $\psi_1(x)<\psi_0(x)$ for all $x\in V$. To this end, we only need to prove $M=\psi_1(x_0)-\psi_0(x_0)=\max_{x\in V}\{\psi_1(x)-\psi_0(x)\}<0$. Suppose not, we have
 $$\D(\psi_1-\psi_0)(x_0)\geq K(\psi_1-\psi_0)(x_0)\geq 0.$$
 Hence $(\psi_1-\psi_0)(y)=(\psi_1-\psi_0)(x_0)$ if $y\sim x_0$. By the connectedness of $G$, we conclude
 $(\psi_1-\psi_0)(y)=(\psi_1-\psi_0)(x_0)$ for any $y\in V$. This contradicts (\ref{ps1}) at $p_j\in V$ and confirms our claim.

 Suppose that $$\psi_0>\psi_1>\cdot\cdot\cdot> \psi_k.$$
 Noting $K>2\l$, we have
\bna\begin{aligned}
(\D-K)(\psi_{k+1}-\psi_k)&=-\l\le[e^{g(\up_0+\psi_k)}\le(e^{g(\up_0+\psi_k)}-1\ri)^2-e^{g(\up_0+\psi_{k-1})}\le(e^{g(\up_0+\psi_{k-1})}-1\ri)^2\ri]-K(\psi_k-\psi_{k-1})\\
&=-\l g'(\xi)e^{g(\xi)}\le(3e^{g(\xi)}-1\ri)\le( e^{g(\xi)}-1\ri)(\psi_k-\psi_{k-1})-K(\psi_k-\psi_{k-1})\\
&=\le[\l e^{g(\xi)}\le( 3e^{g(\xi)}-1\ri)-K\ri](\psi_k-\psi_{k-1})\\
&\geq (2\l-K)(\psi_k-\psi_{k-1})\\
&>0,
\end{aligned}
\ena
 where we have used the mean value theorem, $\up_0+\psi_k\leq \xi\leq \up_0+\psi_{k-1}$, $g'(\xi)=1/(1-e^{g(\xi)})$.
 Then Lemma 2.1 yields $\psi_{k+1}\leq\psi_k$. Using the same argument as in proving $\psi_1<\psi_0$, we get $\psi_{k+1}<\psi_k$.

 Next we prove that $\psi_k>\up_-$. For $k=0$, we have
 \be\label{ps2}\begin{aligned}
\D(\up_--\psi_0)&\geq -\l e^{g(\up_--\psi_{0})}\le(e^{g(\up_--\psi_{0})}-1\ri)^2+4\pi \sum\limits_{j=1}^{N}\delta_{p_j}\\
&=-2\l e^{g(\up_--\psi_{0})}g'(\xi)e^{g(\xi)}(e^{g(\xi)}-1)(\up_--\psi_0)+4\pi \sum\limits_{j=1}^{N}\delta_{p_j}\\
&=2\l e^{g(\up_--\psi_{0})}e^{g(\xi)}(\up_--\psi_0)+4\pi \sum\limits_{j=1}^{N}\delta_{p_j},
\end{aligned}
\ee
where $\xi$ is between $\up_--\psi_0$ and $0$.   Let $\up_-(x_0)-\psi_0(x_0)=\max_{x\in V}\{\up_-(x)-\psi_0(x)\}$. If $(\up_-(x_0)-\psi_0(x_0))\geq 0$, then
$$\D(\up_--\psi_0)(x_0)\geq 0.$$  Hence $(\up_--\psi_0)(y)=(\up_--\psi_0)(x_0)$ if $y\sim x_0$. Noting $G$ is connected, we obtain $(\up_--\psi_0)(x)\equiv (\up_--\psi_0)(x_0)$, which contradicts (\ref{ps2}) at $p_j$. As a conclusion, $\up_-<\psi_0$.

 Suppose that $\up_-<\psi_k$ for any $k\geq 0$. Then we have
\bna\begin{aligned}
(\D-K)(\up_--\psi_{k+1})&\geq-\l\le[e^{g(\up_0+\up_-)}\le(e^{g(\up_0+\up_-)}-1\ri)^2-e^{g(\up_0+\psi_{k})}\le(e^{g(\up_0+\psi_{k})}-1\ri)^2\ri]-K(\up_--\psi_{k})\\
&=\le[\l e^{g(\xi)}\le( 3e^{g(\xi)}-1\ri)-K\ri](\up_--\psi_{k})\\
&\geq (2\l-K)(\up_--\psi_{k})\\
&>0,
\end{aligned}
\ena
where $\up_-+\up_0\leq \xi\leq \psi_k+\up_0$. By the same argument as before, we derive $\up_-<\psi_{k+1}$.  The proof is completed.
\end{proof}

\begin{lemma}
When $\l$ is sufficient big,  Eq.(\ref{cs2}) has a solution.
\end{lemma}

\begin{proof}
By Lemma 3.1, if Eq.(\ref{cs2})  has a lower solution, then $\{\psi_n\}$ converges to a solution of (\ref{cs2}).  Let $
\up_0$ be the solution of (\ref{up1}). Noting $\up_0$ is bounded on $V$, we can choose a positive constant $c'$ such that $\up_0-c'<0$.
Let $\up_-=-c'$ and pick up a sufficient big $\l$ such that $$-\l e^{g(\up_0+\up_-)}\le(e^{g(\up_0+\up_-)}-1\ri)^2+\f{4\pi N}{\ve V\ve}<0.$$
It follows that
$$\D \up_- \geq -\l e^{g(\up_0+\up_-)}\le(e^{g(\up_0+\up_-)}-1\ri)^2+\f{4\pi N}{\ve V\ve}.$$
 Hence $\up_-$ is a lower solution of (\ref{cs2}). The proof is finished.
\end{proof}

 \begin{lemma}
 There exists a critical value $\l_c$ satisfying
 $$\l_c\geq \f{27\pi N}{\ve V\ve},$$
  such that if $\l> \l_c$, Eq.(\ref{cs2}) has a solution and Eq.(\ref{cs2})  has no solution if $\l<\l_c$.
 \end{lemma}
 \begin{proof}
 Define
 $$\Lambda=\Big\{\l>0\,\big|  \text{ $\l $ is  such  that  (\ref{cs2})  has a solution}\Big\}.$$
 We will prove that $\L$ is an interval. Letting $\l'$ be in $\L$, we only need to show
 $$[\l',+\infty)\subset \L.$$
 Suppose that $\up'$ is the solution of Eq.(\ref{cs2}) at $\l=\l'$. Observing $\up_0+\up'<0$, we obtain
 $$\D \up' \geq -\l e^{g(\up_0+\up')}\le(e^{g(\up_0+\up')}-1\ri)^2+\f{4\pi N}{\ve V\ve},$$
 if $\l>\l'$. Hence $\up'$ is a lower solution of (\ref{cs2}) for $\l>\l'$. By Lemma 3.1,  Eq.(\ref{cs2}) has a solution. It yields that $[\l',+\infty)\subset \L$.
 Denote
 $\l_c=\inf\{\l\,|\,\l\in \L\}$. Let $h(t)=-e^t(e^t-1)^2$, $t\in (-\infty, 0]$. We can see that $h$ has a minimal value $-\f{4}{27}$. Hence
 \be\label{le1}
 \D \up\geq -\f{4}{27}\l+\f{4\pi N }{\ve V \ve },\ee if $\up$ satisfies (\ref{cs2}). Integrating (\ref{le1}) over $V$, we obtain
 $$\l\geq\f {27\pi N}{\ve V\ve}.$$  Passing to the limit $\l\ra\l_c$, we have
 $$\l_c\geq \f{27\pi N}{\ve  V\ve}.$$
 \end{proof}

Next we will prove that if $\l=\l_c$,  Eq.(\ref{cs2}) has a solution. First we show the monotonicity of the solution of (\ref{cs2}) with respect to $\l$.
 \begin{lemma}
 Let $\{\up_{\l}\,|\,\l>\l_c\}$ be the  family of maximal solutions of (\ref{cs2}) obtained by Lemma 3.2. Then there holds $\up_{\l_1}>\up_{\l_2}$
if $\l_1>\l_2>\l_c$ .
\end{lemma}
\begin{proof}
Suppose that $\up_{\l}$ is a solution of (\ref{cs2}) and $\l_1>\l_2$. We know $\up_0+\up_{\l}<0$.   It follows from (\ref{cs2}) that

\bna\begin{aligned}
\D \up_{\l_2}&=-\l_2e^{g(\up_0+\up_{\l_2})}\le(e^{g(\up_0+\up_{\l_2})}-1\ri)^2+\f{4\pi N}{\ve V\ve}\\
&=- \l_1e^{g(\up_0+\up_{\l_2})}\le(e^{g(\up_0+\up_{\l_2})}-1\ri)^2+\f{4\pi N}{\ve V\ve}\\
&\,\,\,\,\,+(\l_1-\l_2)e^{g(\up_0+\up_{\l_2})}\le(e^{g(\up_0+\up_{\l_2})}-1\ri)^2\\
&\geq- \l_1e^{g(\up_0+\up_{\l_2})}\le(e^{g(\up_0+\up_{\l_2})}-1\ri)^2+\f{4\pi N}{\ve V\ve},
\end{aligned}\ena
which implies that $\up_{\l_2}$ is the lower solution of (\ref{cs2}) with $\l=\l_1$.  By Lemma 3.1, we know that $\up_{\l_1}\geq \up_{\l_2}$.
We also have
\be\label{mo1}\begin{split}
\D (\up_{\l_2}-\up_{\l_1})&=-\l_2e^{g(\up_0+\up_{\l_2})}\le(e^{g(\up_0+\up_{\l_2})}-1\ri)^2+\l_1e^{g(\up_0+\up_{\l_1})}\le(e^{g(\up_0+\up_{\l_1})}-1\ri)^2\\
&=- \l_1e^{g(\up_0+\up_{\l_2})}\le(e^{g(\up_0+\up_{\l_2})}-1\ri)^2+\l_1e^{g(\up_0+\up_{\l_1})}\le(e^{g(\up_0+\up_{\l_1})}-1\ri)^2\\
&\,\,\,\,\,+(\l_1-\l_2)e^{g(\up_0+\up_{\l_2})}\le(e^{g(\up_0+\up_{\l_2})}-1\ri)^2\\
&>-\l_1\le[e^{g(\up_0+\up_{\l_2})}\le(e^{g(\up_0+\up_{\l_2})}-1\ri)^2-e^{g(\up_0+\up_{\l_1})}\le(e^{g(\up_0+\up_{\l_1})}-1\ri)^2\ri] \\
&=\le[\l_1 e^{g(\xi)}\le( 3e^{g(\xi)}-1\ri)\ri](\up_{\l_2}-\up_{\l_1})\\
&\geq 2\l_1(\up_{\l_2}-\up_{\l_1}),
\end{split}
\ee
where $\up_0+\up_{\l_2}\leq \xi\leq \up_0+\up_{\l_1}$.
If $\up_{\l_1}(x_0)-\up_{\l_2}(x_0)=\min_{x\in V}\{\up_{\l_1}(x)-\up_{\l_2}(x)\}=0$, then by (\ref{mo1}) we have
 $$\D(\up_{\l_1}-\up_{\l_2})(x_0)<0,$$
 which is impossible. Therefore $\up_{\l_{1}}>\up_{\l_2}$ if $\l_1>\l_2>\l_c$.
\end{proof}

\begin{lemma}
Assume that $\up_{\l}$ is a maximal solution of (\ref{cs2}).  We decompose  $\up_{\l}$  as $\up_{\l}=\bar{\up}_{\l}+\up'_{\l}$, where $\bar{\up}_{\l}=\f{1}{\ve V\ve}\int_{V}\up_{\l} d\mu$ and $\up'_{\l}=\up_{\l}-\bar{\up}_{\l}$. Then we get
$$\V \na\up'_{\l}\V_2\leq C\l, $$
where $C$ is a positive constant depending only on $V$.  Furthermore we  derive
$$\ve \bar{\up}_{\l}\ve\leq C(1+\l+\l^2)$$
and
$$\V \up_{\l}\V_{W^{1,2}(V)}\leq C(1+\l+\l^2).$$
\end{lemma}
\begin{proof}
Since $\up_{\l}=\bar{\up}_{\l}+\up'_{\l}$, multiplying Eq.(\ref{cs2}) by $\upl$ and integrating over $V$, we have
\bna\begin{split}
\V\na\up'_{\l}\V_2^2&=\l\iv e^{g(\up_0+\up_{\l})}\le(e^{g(\up_0+\up_{\l})}-1\ri)^2\upl d\mu\\
&\leq \l\iv \ve\up'_{\l}\ve d\mu\leq C\l\ve V\ve^{1/2}\V\na\upl\V_2,
\end{split}\ena
where we have used Lemma 2.2. Hence
\be\label{na}
\V\na \up'_{\l}\V_2\leq C\l.
\ee
Noting $\up_0+\up_{\l}=\up_0+\bar{\up}_{\l}+\up'_{\l}<0$,
we obtain
\be\label{upb}
\bar{\up}_{\l}<-\f{1}{\ve V\ve}\iv \up_0(x)d\mu,
\ee
which gives an upper bound of $\bar{\up}_{\l}$.

Next we will find a lower bound of $\bar{\up}_{\l}$.  In view of $\up_0+\up_{\l}<0$ and (\ref{cs2}), we have
$$\D \up_{\l}+\l e^{g(\up_0+\up_{\l})}\le(1-e^{g(\up_0+\up_{\l})}\ri)\geq \f{4\pi N}{\ve V\ve}.$$
Integrating the above inequality over $V$, we obtain
 \be\label{lb}\l \iv e^{g(\up_0+\up_{\l})}d\mu\geq \l \iv e^{2g(\up_0+\up_{\l})}d\mu+4\pi N>4\pi N.\ee
Since $g(t)$ is also an increasing function with $\lim\limits_{t\ra -\infty}g(t)=-\infty, $
we have
$$\lim\limits_{t\ra -\infty}\f{g(t)}{t}=\lim\limits_{t\ra -\infty}g'(t)=\lim\limits_{t\ra -\infty}\f{1}{1-e^{g(t)}}=1.$$
Hence there is positive constant $M$ such that
$$g(t)\leq \f{t}{2}, $$
if $t<-M$.
We decompose $V$ as
$V=V_1\cup V_2,$ where
$$V_1=\{x\in V\,|\,\up_0+\up_{\l}<-M\},\,\,\,V_2=\{x\in V\,|\,-M\leq \up_0+\up_{\l}<0\}.$$

Noting that $g(t)$ is bounded on $[-M,0]$, there exists a positive constant $C$ such that
$$\ve g(t)\ve \leq C,$$ if $t\in [-M,0]$.  Here and after we use $C$ to denote a variable constant.

If $x\in V_1$, we have
\bna
e^{g(\up_0+\up_{\l})}\leq e^{\f{\up_0+\up_{\l}}{2}} \leq e^{\f{\up_0+\up_{\l}}{2}+2C}.
\ena

If $x\in V_2$, we enlarge $C$ such that $C\geq \f{M}{2}$, then
\bna
e^{g(\up_0+\up_{\l})}\leq e^{C} \leq e^{\f{\up_0+\up_{\l}}{2}+2C}.
\ena

It follows that
\be\label{lo}
\begin{split}
\iv e^{g(\up_0+\up_{\l})}d\mu &\leq \iv e^{\f{\up_0+\up_{\l}}{2}+2C}d\mu\\
&=e^{2C}e^{\f{\bar{\up}_{\l}}{2}}\iv e^{\f{\up_0+\up'_{\l}}{2}}d\mu\\
&\leq C e^{\f{\bar{\up}_{\l}}{2}} \le(\iv e^{\up_0}d\mu\ri)^{1/2}\le(\iv e^{\up'_{\l}}d\mu\ri)^{1/2}\\
&\leq C e^{\f{\bar{\up}_{\l}}{2}}\le(\iv e^{\up'_{\l}}d\mu\ri)^{1/2}\\
&= C e^{\f{\bar{\up}_{\l}}{2}}\le(\iv e^{\V \na\upl\V_2\f{\up'_{\l}}{\V\na \upl\V_2}}d\mu\ri)^{1/2}\\
&\leq C e^{\f{\bar{\up}_{\l}}{2}}\le(\iv e^{\V \na\upl\V_2^2+\f{|\up'_{\l}|^2}{4\V \na\upl\V_2^2}}d\mu\ri)^{1/2}\\
&\leq  C e^{\f{\bar{\up}_{\l}}{2}}e^{\f{\V \na\upl\V_2^2}{2}},
\end{split}
\ee
where we have used Lemma $\ref{lm3}$.
It follows from that (\ref{lb}) and (\ref{lo}) that

$$e^{\f{\bar{\up}_{\l}}{2}}\geq C\l^{-1}e^{-\f{\V \na\upl\V_2^2}{2}}.$$
This together with (\ref{na}) and (\ref{upb}) yields
$$\ve \bar{\up}_{\l}\ve\leq C(1+\l+\l^2).$$
In addition,
$$\V \up_{\l}\V_{W^{1,2}(V)}\leq C(1+\l+\l^2).$$
\end{proof}

\begin{lemma}
The equation (\ref{cs2}) has a solution at $\l=\l_c$.

\end{lemma}
\begin{proof}
Set $\l_c<\l<\l_c+1$. Then Lemma 3.5 implies that $\{\up_{\l}\}$ is bounded in $W^{1,2}(V)$. Noting that $W^{1,2}(V)$ is pre-compact and
$\{\up_{\l}\}$ is increasing with respect to $\l$, we have
$$\up_{\l}\ra \tilde{\up}\,\,\text{in}\,\,W^{1,2}(V),$$
as $\l\ra \l_c$ and
$$\up_0+\tilde{\up}_{\l}<0,\,\,\text{in}\,\,V.$$
Since the convergence of $\{\up_{\l}\}$ is pointwise, we have
$$\D \up_{\l}\ra \D\tilde{\up}$$
and
$$e^{g(\up_0+\up_{\l})}\le(e^{g(\up_0+\up_{\l})}-1\ri)^2\ra e^{g(\up_0+\tilde{\up})}\le(e^{g(\up_0+\tilde{\up})}-1\ri)^2,$$ as $\l\ra \l_c$.
Therefore $\tilde{\up}$ is solution of (\ref{cs2}) and the proof is completed.
\end{proof}

\section{Proof of Theorem 1.2}

Let $\up_0$ be the solution of (\ref{up1}) and $u=\up_0+\up$. Then  we  rewrite
$$\D u=\l e^u(e^u-1)+4\pi\sum\limits_{j=1}^{N}\delta_{p_j}$$
 as
\be\label{mfe}
\D \up=\l e^{\up_0+\up} (e^{\up_0+\up}-1)+\f{4\pi N}{\ve V\ve}.
\ee
Lemma 4.2 and Lemma 4.4 in \cite{HLY} imply that if (\ref{mfe}) has a lower solution, then it has a maximal solution which is negative.
 \begin{lemma}
 Let $\{\up_{\l}\,|\,\l>\l_c\}$ be the family of maximal solutions of (\ref{mfe}). Then we get $\up_{\l_1}>\up_{\l_2}$
if $\l_1>\l_2>\l_c$ .
\end{lemma}
\begin{proof}
Suppose $\l_1>\l_2$. Noting $\up_{\l}+\up_0<0$,  we have

\bna\begin{aligned}
\D \up_{\l_2}&=\l_2e^{\up_0+\up_{\l_2}}\le(e^{\up_0+\up_{\l_2}}-1\ri)+\f{4\pi N}{\ve V\ve}\\
&= \l_1e^{\up_0+\up_{\l_2}}\le(e^{\up_0+\up_{\l_2}}-1\ri)+\f{4\pi N}{\ve V\ve}\\
&\,\,\,\,\,+(\l_1-\l_2)e^{\up_0+\up_{\l_2}}\le(1-e^{\up_0+\up_{\l_2}}\ri)\\
&\geq\l_1e^{\up_0+\up_{\l_2}}\le(e^{\up_0+\up_{\l_2}}-1\ri)+\f{4\pi N}{\ve V\ve},
\end{aligned}\ena
which implies that $\up_{\l_2}$ is lower solution of (\ref{mfe}) with $\l=\l_1$.  Hence $\up_{\l_1}\geq \up_{\l_2}$ by Lemma 4.2 in \cite{HLY}.
Calculate
\be\label{mo2}\begin{split}
\D (\up_{\l_2}-\up_{\l_1})&=\l_2e^{\up_0+\up_{\l_2}}\le(e^{\up_0+\up_{\l_2}}-1\ri)-\l_1e^{\up_0+\up_{\l_1}}\le(e^{\up_0+\up_{\l_1}}-1\ri)\\
&= \l_1e^{\up_0+\up_{\l_2}}\le(e^{\up_0+\up_{\l_2}}-1\ri)-\l_1e^{\up_0+\up_{\l_1}}\le(e^{\up_0+\up_{\l_1}}-1\ri)\\
&\,\,\,\,\,+(\l_1-\l_2)e^{\up_0+\up_{\l_2}}\le(1-e^{\up_0+\up_{\l_2}}\ri)\\
&>\l_1\le[e^{\up_0+\up_{\l_2}}\le(e^{\up_0+\up_{\l_2}}-1\ri)-e^{\up_0+\up_{\l_1}}\le(e^{\up_0+\up_{\l_1}}-1\ri)\ri] \\
&=\l_1e^{2\up_0}(e^{2\up_{\l_2}}-e^{2\up_{\l_1}}) -\l_1e^{\up_0}(e^{\up_{\l_2}}-e^{\up_{\l_1}})   \\
&\geq 2\l_1e^{2\up_0+2\xi}(\up_{\l_2}-\up_{\l_1}),
\end{split}
\ee
where $$\up_{\l_2}\leq \xi\leq \up_{\l_1}.$$
Hence
$$\D (\up_{\l_1}-\up_{\l_2})<2\l_1e^{2\up_0+2\xi}(\up_{\l_1}-\up_{\l_2}).$$
Supposing that $\up_{\l_1}(x_0)-\up_{\l_2}(x_0)=\min_{x\in V}\{\up_{\l_1}(x)-\up_{\l_2}(x)\}=0$, then we have
 $$\D(\up_{\l_1}-\up_{\l_2})(x_0)<0,$$
 which contradicts $\D(\up_{\l_1}-\up_{\l_2})(x_0)\geq 0$. Therefore, we conclude that $\up_{\l_{1}}>\up_{\l_2}$ if $\l_1>\l_2>\l_c$.

\end{proof}

\begin{lemma}
For  a maximal solution $\up_{\l}$  of (\ref{mfe}), we decompose  it as $\up_{\l}=\bar{\up}_{\l}+\up'_{\l}$, where $\up'_{\l}=\up_{\l}-\bar{\up}_{\l}$. Then we have
$$\V \na\up'_{\l}\V_2\leq C\l, $$
where $C$ is a positive constant depending only on $V$.  Furthermore we obtain
$$\ve \bar{\up}_{\l}\ve\leq C(1+\l+\l^2)$$
and
$$\V \up_{\l}\V_{W^{1,2}(V)}\leq C(1+\l+\l^2),$$ where $C$ is a positive constant depending only on $G$.
\end{lemma}
\begin{proof}

In view of  $\up_{\l}=\bar{\up}_{\l}+\up'_{\l}$,   multiplying Eq.(\ref{mfe}) by $\upl$ and integrating over $V$, we have
\bna\begin{aligned}
\V\na\up'_{\l}\V_2^2&=\l\iv e^{\up_0+\up_{\l}}\le(1-e^{\up_0+\up_{\l}}\ri)\upl d\mu\\
&\leq \l\iv \ve\up'_{\l}\ve d\mu\leq C\l\ve V\ve^{1/2}\V\na\upl\V_2.
\end{aligned}
\ena
It follows that \be\label{na2}
\V \na\up'_{\l}\V_2\leq C\l.
\ee
Using $\up_0+\up_{\l}=\up_0+\bar{\up}_{\l}+\up'_{\l}<0$, we obtain
\be\label{up2}\bar{\up}_{\l}<-\f{1}{\ve V\ve}\iv \up_0(x)d\mu.\ee
Integrating Eq.(\ref{mfe}) with $\up=\up_{\l}$ over $V$, we have
 \be\label{lb2}\l \iv e^{\up_0+\up_{\l}}d\mu= \l \iv e^{2(\up_0+\up_{\l})}d\mu+4\pi N>4\pi N.\ee
 By the H\"{o}lder inequality and Lemma $\ref{lm3}$, we obtain
 \be\label{lo2}
\begin{split}
\iv e^{\up_0+\up_{\l}}d\mu &= \iv e^{\up_0+\bar{\up}_{\l}+\upl}d\mu\\
&=e^{\bar{\up}_{\l}}\iv e^{\up_0+\up'_{\l}}d\mu\\
&\leq  e^{\bar{\up}_{\l}} \le(\iv e^{2\up_0}d\mu\ri)^{1/2}\le(\iv e^{2\up'_{\l}}d\mu\ri)^{1/2}\\
&\leq C e^{\bar{\up}_{\l}}\le(\iv e^{2\up'_{\l}}d\mu\ri)^{1/2}\\
&= C e^{\bar{\up}_{\l}}\le(\iv e^{2\V \na\upl\V_2\f{\up'_{\l}}{\V \na\upl\V_2}}d\mu\ri)^{1/2}\\
&\leq C e^{\bar{\up}_{\l}}\le(\iv e^{2\V \na\upl\V_2^2+\f{|\up'_{\l}|^2}{2\V \na\upl\V_2^2}}d\mu\ri)^{1/2}\\
&\leq  C e^{\bar{\up}_{\l}}e^{\V\na \upl\V_2^2},
\end{split}
\ee
which  together with (\ref{lb2}) leads to
\be\label{lei} e^{\bar{\up}_{\l}}\geq C\l^{-1}e^{-\V \na\upl\V_2^2}.\ee
Combining (\ref{na2}), (\ref{up2}) and (\ref{lei}), we get
$$\ve \bar{\up}_{\l}\ve\leq C(1+\l+\l^2).$$
Moreover, $$\V \up_{\l}\V_{W^{1,2}(V)}\leq C(1+\l+\l^2).$$
\end{proof}
\begin{lemma}
The equation (\ref{mfe}) with $\l=\l_c$ has a solution.

\end{lemma}
\begin{proof}
Let  $\l_c<\l<\l_c+1$. By  Lemma 4.2, we have $\{\up_{\l}\}$ is bounded in $W^{1,2}(V)$. Since  $W^{1,2}(V)$ is pre-compact and
 $\{\up_{\l}\}$ is increasing with respect to $\l$, we obtain
$$\up_{\l}\ra \tilde{\up}\,\,\text{in}\,\,W^{1,2}(V),$$
as $\l\ra \l_c$ and
$$\up_0+\tilde{\up}_{\l}<0,\,\,\text{in}\,\,V.$$
In view of  the pointwise convergence of $\{\up_{\l}\}$,  we have
$$\D \up_{\l}\ra \D\tilde{\up}$$
and
$$e^{\up_0+\up_{\l}}\le(e^{\up_0+\up_{\l}}-1\ri)\ra e^{\up_0+\tilde{\up}}\le(e^{\up_0+\tilde{\up}}-1\ri), $$as $\l\ra \l_c$.
It is clear that  $\tilde{\up}$ is solution of (\ref{mfe}).
\end{proof}

\vskip 30 pt
\noindent{\bf Acknowledgement}

This work is  partially  supported by the National Natural Science Foundation of China (Grant No. 11721101), and by National Key Research and Development Project SQ2020YFA070080.


\vskip 30 pt















 \begin{thebibliography}{10}



\bibitem{BCT92} J. Burzlaff, A. Chakrabarti, D.H. Tchrakian, Generalized self-dual Chern-Simons vortices, Phys. Lett. B 293 (1992) 127-131.

\bibitem{CY95} L.A. Caffarelli,  Y. Yang,  Vortex condensation in the Chern-Simons Higgs model: an existence theorem,  Commun. Math. Phys. 168 (1995)  321-336.
\bibitem{AE12} A.Elmoataz, etal., Non-local morphological PDEs and $p$-Laplacian
equation on graphs with applications in image processing and
machine learning, IEEE J.Sel.Top.Signal Process.(2012).
\bibitem{Ge20} Huabin Ge, The $p$th Kazdan-Warner equation on graphs, Communications in Contemporary Mathematics (2020) 1950052.

\bibitem{GJ18} Huabin Ge, Wenfeng Jiang, Kazdan-Warner equation on infinite graphs, J. Korean Math. Soc. 55(5) (2018)  1091-1101.
\bibitem{GHJ} Huabin Ge, Bobo Hua, Wenfeng Jiang,  A note on Liouville type equations on graphs, Proc. Amer. Math. Soc. 146(11) (2018) 4837-4842.
\bibitem{LJ10}L.J. Grady, etal.,  Discrete Calculus, Applied Analysis on Graphs for
Computational Science,Springer, 2010.

\bibitem{GLY16} Alexander Grigor'yan, Yong Lin,  Yunyan Yang, Yamabe type equations on graphs, J. Differential Equations 261 (2016) 4924-4943.
\bibitem{GLY17} Alexander Grigor'yan, Yong Lin,  Yunyan Yang, Existence of positive solutions to some nonlinear
equations on locally finite graphs, Science China Mathematics 60(7) (2017) 1311-1324.
\bibitem{GLY}  Alexander Grigor'yan, Yong Lin,  Yunyan Yang, Kazdan-Warner equation on graph, Calc. Var. Partial Differential Equations 55(4) (2016) Art. 92 13 pp.
\bibitem{HLY} An Huang, Yong Lin , Shing-Tung Yau,  Existence of solutions to mean field equations on graphs,  Comm. Math. Phys. 377(1) (2020)  613-621.
\bibitem{HSZ} Xiaoli Han, Mengqiu Shao, Liang Zhao, Existence and convergence of solutions for nonlinear biharmonic equations on graphs, J. Differential Equations 268 (2020) 3936-3961.
\bibitem{Han14} Xiaosen Han,  Existence of doubly periodic vortices in a generalized Chern-Simons model, Nonlinear Anal. Real World Appl. 16 (2014) 90-102.
\bibitem{Hou21} Songbo Hou, Multiple solutions of a nonlinear biharmonic equation on graphs, Communications in Mathematics and Statistics, to appear.
\bibitem{HKP90} J. Hong, Y. Kim, P. Y. Pac, Multivortex solutions of the abelian Chern-Simons-Higgs
theory, Phys. Rev. Lett. 64(19) (1990)  2230-2233.

\bibitem{JW90}  R.W. Jackiw, E.J. Weinberg, Self-dual Chen-Simons vortices, Phys. Rev. Lett. 64 (1990) 2234-2237.
\bibitem{LZ} Yingshu L\"{u}, Peirong Zhong, Existence of solutions to a generalized self-dual Chern-Simons equation on graphs, arXiv: 2107.12535.
\bibitem{KS18}Matthias Keller,  Michael Schwarz,  The Kazdan-Warner equation on canonically compactifiable graphs,  Calc. Var. Partial Differential Equations 57(2) (2018)  Art. 70 18 pp.

\bibitem{GT96} Gabriella Tarantello, Multiple condensate solutions for the Chern-Simons-Higgs theory, J. Math. Phys. 37(8) (1996) 3769-3796.
\bibitem {LW18}Yong Lin, Yiting Wu, Yiting, Blow-up problems for nonlinear parabolic equations on locally finite graphs, Acta Math. Sci. Ser. B (Engl. Ed.) 38(3) (2018) 843-856.


\bibitem{LW17} Yong Lin, Yiting Wu, The existence and nonexistence of global solutions for a semilinear heat equation on graphs, Calc. Var. Partial Differential Equations 56(4) (2017) Art. 102  22 pp.

\bibitem{LY20}Shuang Liu, Yunyan Yang, Multiple solutions of Kazdan-Warner equation on graphs in the negative case, Calc. Var. Partial Differential Equations 59(5) (2020) Art. 164 15pp.
\bibitem{ZZ18} Ning Zhang, Liang Zhao, Convergence of ground state solutions for nonlinear
Schr\"{o}dinger equations on graphs, Sci. China Math. 61(8) (2018) 1481-1494.
\bibitem{ZC18} Xiaoxiao Zhang, Yanxun Chang, $p$-th Kazdan-Warner equation on graph in the negative case, J. Math. Anal. Appl. 466(1) (2018)  400-407.
\bibitem{ZL19} Xiaoxiao Zhang, Aijin  Lin, Positive solutions of $p$-th Yamabe type equations on infinite graphs, Proc. Amer. Math. Soc. 147(4) (2019)  1421-1427.

\end{thebibliography}
\end{document}